\documentclass[12pt]{amsart}

\usepackage{amsmath}
\usepackage{amsthm,amssymb}
\usepackage{enumerate}
\usepackage[utf8]{inputenc}

\usepackage{a4wide}
\usepackage{amsaddr}
\usepackage{hyperref}

\theoremstyle{plain}
\newtheorem{theorem}{Theorem}[section]
\newtheorem{lemma}[theorem]{Lemma}
\newtheorem{prop}[theorem]{Proposition}
\newtheorem{corollary}[theorem]{Corollary}
\theoremstyle{definition}

\newtheorem{exmp}[theorem]{Example}

\newtheorem{remark}[theorem]{Remark}

\newcommand{\Nil}{\mbox{\rm Nil}}

\newcommand{\supp}{\mbox{\rm supp}}

\newcommand{\Z}{\mathbb{Z}}

\newcommand{\R}{\mathbb{R}}

\newcommand{\Cal}[1]{{\mathcal #1}}
\DeclareMathOperator{\id}{id}

\author{Mai Hoang Bien$^{1,2}$ and Johan \"{O}inert$^2$}

\address{$^{1}$Department of Basic Sciences, University of Architecture, 196 Pasteur Str., Dist. 1, Ho Chi Minh City, Vietnam}
\email{maihoangbien012@yahoo.com}

\address{$^{2}$Blekinge Institute of Technology, Department of Mathematics and Natural Sciences, SE-37179 Karlskrona, Sweden}
\email{johan.oinert@bth.se}

\thanks{This research was supported by the Crafoord Foundation, research grant no. 20150871.}
\date{2017-01-01}


\begin{document}

\begin{abstract}
In this article we give a characterization of left (right) quasi-duo differential polynomial rings.
In particular, we show that a differential polynomial ring is left quasi-duo if and only if it is right quasi-duo.
This yields a partial answer to a question posed by Lam and Dugas in 2005.
We provide non-trivial examples of such rings and give a complete description of the maximal ideals of an arbitrary quasi-duo differential polynomial ring.
Moreover, we show that there is no left (right) quasi-duo differential polynomial ring in several indeterminates.
\end{abstract}

\title{Quasi-duo differential polynomial rings}
\maketitle


\section{Introduction}

Throughout this article, all rings are assumed to be unital and associative.
Following \cite{Feller_58},
a 
ring $S$ is said to be \emph{left (right) duo}
if every left (right) ideal of $S$ is a two-sided ideal.
More generally, $S$ is said to be \emph{left (right) quasi-duo}
if every maximal left (right) ideal of $S$ is a two-sided ideal (see e.g. \cite{Yu95})
or, equivalently, if every left (right) primitive homomorphic image of $S$ is a division ring (see e.g. \cite[Proposition 4]{Pu15}).
A ring which is both left and right quasi-duo is called \emph{quasi-duo}.

Quasi-duo rings appear in various places in ring theory, e.g. in the investigation of the K\"{o}the conjecture (see e.g. \cite[Proposition 2.5]{Pu06}).
There are many open problems concerning left (right) quasi-duo rings, one of which is due to Lam and Dugas who ask whether there exists a right quasi-duo ring
which is not left quasi-duo (see \cite[Question 7.7]{Pa_LaDu_05}).
Our main result (Theorem~\ref{MainResult}) shows that such an example can not be found in the class of differential polynomial rings.

Recall that an \emph{Ore extension} $R[x;\sigma,\delta]$ is constructed from a 
ring $R$, a ring endomorphism $\sigma : R\to R$ (respecting $1_R$)
and a $\sigma$-derivation $\delta : R \to R$,
i.e. an additive map satisfying
\begin{displaymath}
	\delta(rs)=\sigma(r)\delta(s)+\delta(r)s, \quad \forall r,s\in R.
\end{displaymath}
As a left $R$-module $R[x;\sigma,\delta]$ is equal to the usual polynomial ring $R[x]$.
The multiplication on $R[x;\sigma,\delta]$ is defined by the rule
\begin{displaymath}
	xr = \sigma(r)x + \delta(r)
\end{displaymath}
for $r\in R$. This turns the Ore extension $R[x;\sigma,\delta]$ into a unital and associative ring (see e.g. \cite{Ny13}).
If $\delta=0$, then $R[x;\sigma,0]$ is said to be a \emph{skew polynomial ring}.
If, on the other hand, $\sigma=\id_R$, then $\delta$ is called a \emph{derivation}
and
$R[x;\id_R,\delta]$ is said to be a \emph{differential polynomial ring}
and will simply be denoted by $R[x;\delta]$.

In \cite{Pa_LeMaPu_08}, Leroy, Matczuk and Puczylowski
obtained a complete characterization of left (right) quasi-duo skew polynomial rings.
In \cite{Le10}, the same authors continued their investigation and gave a complete characterization of left (right) quasi-duo $\Z$-graded rings.

In this article we direct our attention to another type of Ore extensions, namely the differential polynomial rings.
Our main result is the following.

\begin{theorem}\label{MainResult}
Let $S=R[x;\delta]$ be a differential polynomial ring, and put $J_0 = J(S) \cap R$.
The following five assertions are equivalent:
	\begin{enumerate}[{\rm (i)}]
		\item $S$ is left quasi-duo;
		\item $S$ is right quasi-duo;
		\item Every left ideal of $S$ containing the Jacobson radical $J(S)$ is two-sided,
		i.e. $S/J(S)$ is left duo;
		\item Every right ideal of $S$ containing the Jacobson radical $J(S)$ is two-sided, i.e. $S/J(S)$ is right duo;
		\item\label{MainTheorem:Cond5} The quotient ring $R/J_0$ is commutative and $\delta(R)\subseteq J_0$.
	\end{enumerate}
\end{theorem}

This result provides a complete characterization of left (right) quasi-duo differential polynomial rings.
In particular, it shows that a differential polynomial ring is left quasi-duo if and only if it is right quasi-duo,
thereby yielding a partial answer to \cite[Question 7.7]{Pa_LaDu_05}.
Notice that for skew polynomial rings, which were studied in \cite{Pa_LeMaPu_08}, the same question is still wide open.

This article is organized as follows.

In Section~\ref{Sec:Proof} we prove Theorem~\ref{MainResult}.
In Section~\ref{Sec:Examples} we show that if $R[x;\delta]$ is quasi-duo, then if $R$ belongs to certain classes of rings, we can conclude that $R[x;\delta]$ is commutative (see Proposition~\ref{main_3}). This means, in particular, that $\delta=0$ and hence $R[x;\delta]$ is a polynomial ring. We also provide examples of quasi-duo differential polynomial rings which are not polynomial rings (see Example~\ref{e3.5}).
In Section~\ref{Sec:MaximalIdeals} we give a complete description of the maximal ideals of quasi-duo differential polynomial rings (see Theorem~\ref{t4.4}).
In Section~\ref{Sec:SeveralVariables} we consider differential polynomial rings in several indeterminates, $R[X;D]$, defined by a (countable) set of variables $X$ and a family $D$ of derivations on $R$. We show that $R[X;D]$ can never be quasi-duo if $X$ consist of more than one variable (see Theorem~\ref{thm:several}).

\section{Proof of the main result}\label{Sec:Proof}

In this section we give a proof of our main result, Theorem~\ref{MainResult}.
We begin by showing that a left (right) quasi-duo differential polynomial ring over a simple ring
is necessarily a polynomial ring.

\begin{lemma}\label{t2.2}
Let $S=R[x;\delta]$ be a left (right) quasi-duo differential polynomial ring.
If $R$ is a simple ring, then $R$ is a field and $\delta=0$.
\end{lemma}

\begin{proof}
Let $L$ be a maximal left (right) ideal of $S$ containing $x$.
By assumption, $L$ is a two-sided ideal of $S$ and hence, for any $r\in R$,
we get $\delta(r)=xr-rx \in L \cap R$.
Notice that $L\cap R$ is a two-sided ideal of $R$.
Using that $R$ is a simple ring and that $L \neq S$, we conclude that
$L \cap R = \{0\}$ and hence $\delta(r)=0$, for all $r\in R$.

Take $a\in R$ and let $M$ be a maximal left (right) ideal of $S$ containing $x-a$.
By assumption $M$ is a two-sided ideal of $S$ and hence, for any $b\in R$, we get
$ab-ba=b(x-a)-(x-a)b\in M$.
Using the same argument as before we get $M \cap R = \{0\}$ and hence $ab-ba=0$.
This shows that $R$ is a commutative and simple ring, i.e. a field.
\end{proof}

The following proposition gives us a necessary condition on the ring $R$ and the derivation $\delta$ in order for the differential polynomial ring $R[x;\delta]$ to be left (right) quasi-duo.

\begin{prop}\label{Prop:Nec}
Let $S=R[x;\delta]$ be a differential polynomial ring, and put $J_0=J(S)\cap R$.
If $S$ is left (right) quasi-duo, then $R/J_0$ is commutative and $\delta(R)\subseteq J_0$.
\end{prop}

\begin{proof}
Suppose that $S$ is left quasi-duo. (The right quasi-duo case is treated analogously.)
There are now two cases: (1) there exists a maximal left ideal $M$ of $S$ such that $M\cap R=\{0\}$;
or (2) $M\cap R\ne \{0\}$ for any non-zero maximal left ideal $M$ of $S$.
	
{\bf Case 1.} Suppose that there exists a maximal left ideal $M$ of $S$ such that $M\cap R=\{0\}$. In this case, $J_0=J(S)\cap R=\{0\}$. We claim that $R$ is a division ring. Indeed, since $S$ is left quasi-duo, $M$ is two-sided. Hence, the factor ring $K=S/M$ is a division ring. Let $a\in R\backslash\{0\}$. Since $a\notin M$, $\overline a$ is invertible in $K$. Let $a_0,a_1,\cdots,a_n\in R$ such that $\overline{a}(\overline {a_0}+\overline {a_1}x+\cdots+\overline {a_n}x^n)=\overline 1$ in $K$. Then $\overline {aa_0}=\overline{1}$ in $K$. Hence, $aa_0-1\in M$. By the assumption $R\cap M=\{0\}$, we get that $aa_0=1$. This means that every non-zero element of $R$ 
	has a right inverse in $R$,
	which implies that $R$ is a division ring.
	By Lemma~\ref{t2.2} we now conclude that $R$ is a field and that $\delta=0$.
	
{\bf Case 2.} Suppose that for any non-zero maximal left ideal $M$ of $S$, $M\cap R\ne \{0\}$ holds.
	Let $M$ be a non-zero maximal left ideal of $S$. Put $M_0=M\cap R$. 
	For any $a\in M_0$, one has $\delta(a)=xa-ax\in M$.
	Hence, $\delta(a)\in M_0$ for any $a\in M_0$. This means that $M_0$ is $\delta$-invariant. Therefore, $\delta$ induces a derivation $\overline{\delta}$ on $R/M_0$, namely, $\overline{\delta}(\overline{a})=\overline{\delta(a)}$, for $a\in R$.
	Consider the map
	\begin{displaymath}
		\varphi : R[x;\delta]\to (R/M_0)[x;\overline{\delta}],
		\quad a_0+a_1 x +\ldots+a_n x^n \mapsto \overline {a_0}+\overline{a_1} x + \ldots + \overline{a_n} x^n.
	\end{displaymath}
	It is easy to check that $\varphi$ is a surjective ring morphism. Thus, in view of \cite[Page 245]{Pa_LaDu_05}, $(R/M_0)[x;\overline{\delta}]$ is left quasi-duo. Moreover, the ideal $\overline M =\varphi(M)$ is a maximal left ideal of $(R/M_0)[x;\overline{\delta}]$ and $(R/M_0)\cap \overline{M}=\{\overline 0\}$. Now, by applying Case 1, we get that $R/M_0$ is commutative and that $\overline{\delta}=0$.

	It remains to show that $R/J_0$ is commutative and that $\delta(R)\subseteq J_0$ holds.
	To see this, notice that we have already proved that for any maximal left ideal $M$, $\overline \delta(\overline a)=\overline 0$ and $\overline{ab}=\overline{ba}$ for any $\overline{a},\overline{b} \in R/M_0$ where $M_0=M\cap R$. As a corollary, $\delta(a)\in M$ and $ab-ba\in M$, for all $a,b\in R$ and every maximal left ideal $M$ of $S$. Therefore, $\delta(a), ab-ba\in J(S)\cap R=J_0$ for any $a,b\in R$. Thus,  $\delta(R)\subseteq J_0$ and $R/J_0$ is commutative. This concludes the proof.
\end{proof}

We shall now prove Theorem~\ref{MainResult} and thereby get a complete characterization of quasi-duoness of $R[x;\delta]$.\\

\noindent{\bf Proof of Theorem~\ref{MainResult}}\\
We will only prove the left sided case, i.e. (i)$\Leftrightarrow$(iii)$\Leftrightarrow$(v).
The right sided case, i.e. (ii)$\Leftrightarrow$(iv)$\Leftrightarrow$(v), is treated analogously.
We will now show that (i)$\Rightarrow$(v)$\Rightarrow$(iii)$\Rightarrow$(i).
	
\noindent (i)$\Rightarrow$(v): This implication follows from Proposition~\ref{Prop:Nec}.
	
\noindent (v)$\Rightarrow$(iii):
	Consider the morphism $\varphi$ as defined in Case 2 of Proposition~\ref{Prop:Nec}:
	\begin{displaymath}
		\varphi : R[x;\delta]\to (R/J_0)[x;\overline{\delta}], \quad a_0+a_1x+\ldots+a_nx^n\mapsto \overline{a_0}+\overline{a_1}x+\ldots+\overline{a_n}x^n.
	\end{displaymath}
In our case, $R/J_0$ is commutative and $\overline{\delta}=0$. Hence, $(R/J_0)[x,\overline{\delta}]$ is commutative.
Notice that $\varphi$ is surjective and that $\ker (\varphi)=J_0[x,\delta]\subseteq J(S)$. Hence, $ S/(J_0[x,\delta])\cong (R/J_0)[x,\overline{\delta}]$ which is commutative. Therefore, every left ideal of $S$ containing $J(S)$ is two-sided.
	
\noindent (iii)$\Rightarrow$(i): This is trivial.	\qed

\begin{remark}
Theorem~\ref{MainResult} shows that a differential polynomial ring is left quasi-duo if and only if it is right quasi-duo.
Henceforth, we need not make a distinction between the left and the right properties and shall simply use the notion \emph{quasi-duo}.
\end{remark}

\section{Trivial and non-trivial quasi-duo differential polynomial rings}\label{Sec:Examples}

By Lemma~\ref{t2.2} we have observered that if the differential polynomial ring $R[x;\delta]$ is quasi-duo and $R$ is simple, then $R[x;\delta]$ is necessarily a commutative polynomial ring.
In this section we show that the same conclusion holds for large classes of rings $R$ which are not necessarily simple (see Proposition~\ref{main_3}).
We will also show that there exist quasi-duo differential polynomial rings which are not classical polynomial rings (see Example~\ref{e3.5}).

\begin{lemma}\label{l3.1}
Let $S=R[x;\delta]$ be a differential polynomial ring,
and denote the nilradical of $R$ by $\Nil(R)$.
Suppose that $R$ satisfies at least one of the following conditions:
	\begin{enumerate}[{\rm (i)}]
		\item $\Nil(R)=\{0\}$ and $R$ is a PI-ring, i.e. $R$ satisfies a polynomial identity;
		\item $\Nil(R)=\{0\}$ and $R$ satisfies the ascending chain condition on right annihilators.
	\end{enumerate}
Then, $S$ is semiprimitive, i.e. $J(S)=\{0\}$.
\end{lemma}
\begin{proof}
This is just a corollary of \cite{Pa_TsWuCh_07}. 
\end{proof}

By the preceding lemma, the class of rings $R$ over which differential polynomial rings $R[x;\delta]$ are semiprimitive includes e.g. semiprime commutative rings, domains, and noetherian rings with $\Nil(R)=\{0\}$.

\begin{prop}\label{main_3}
Let $R$ be a ring satisfying $\Nil(R)=\{0\}$.
If $R$ is a PI-ring or satisfies the ascending chain condition on right annihilators, then the following two assertions are equivalent:
\begin{enumerate}[{\rm (i)}]
	\item $R[x;\delta]$ is quasi-duo;
	\item $R[x;\delta]$ is commutative.
\end{enumerate}
\end{prop}

\begin{proof}
The desired conclusion follows from Lemma~\ref{l3.1} and Theorem~\ref{MainResult}.
\end{proof}

The following example demonstrates a quasi-duo differential polynomial ring which is non-trivial, i.e. not a polynomial ring.

\begin{exmp}\label{e3.5}
	Let $R=\left\{ 	\begin{bmatrix}
	a & b \\
	0 & c
	\end{bmatrix} \middle| \, a,b,c\in \R \right\}$.
	It is not difficult to see that the Jacobson radical of $R$ is $J(R)=\left\{ 	\begin{bmatrix}
	0 & b \\
	0 & 0
	\end{bmatrix} \middle| \, b\in \R \right\}.$
	Hence, $R/J(R)=\left\{ 	\begin{bmatrix}
	a & 0 \\
	0 & c
	\end{bmatrix} \middle| \, a,c\in \R \right\}$ is commutative.
	Take $A\in J(R)\backslash \{0\}$ and let $\delta$ be the inner derivation on $R$ defined by $\delta(B)=AB-BA$, for $B\in R$.
	Clearly, $\delta(R)\subseteq J(R)$. Now consider the corresponding differential polynomial ring $R[x;\delta]$.
	Using that $R$ is a PI-ring over $\R$, a field of characteristic zero, \cite[Theorem 1.2]{Pa_BeMaSh_16}
	yields that $J(R[x;\delta])=\Nil(R)[x;\delta]$.
	Since every element of $R$ is a root of a non-zero polynomial over $\R$, by \cite[Corollary 4.19]{Bo_La_91}, $J(R)=\Nil(R)$. This means that $J_0=J(R[x;\delta])\cap R=J(R)$. Therefore, $R[x;\delta]$ satisfies Theorem~\ref{MainResult}\eqref{MainTheorem:Cond5} and hence $R[x;\delta]$ is left and right quasi-duo.
\end{exmp}

\begin{remark}
One may replace $R$ by an arbitrary ring of upper triangular $n$ by $n$ matrices
and mimick the above construction.
This gives us a whole family of non-trivial quasi-duo differential polynomial ring.
\end{remark}

\section{Maximal ideals and the Jacobson radical}\label{Sec:MaximalIdeals}

In this section, we shall describe all maximal ideals of $S=R[x;\delta]$ in the case when $S$ is quasi-duo. Notice that in \cite{Pa_LeMaPu_08}, Leroy et al. gave a complete characterization of left (right) quasi-duo skew polynomial rings $R[x;\sigma,0]$ where $\sigma$ is an automorphism of $R$. In order to do so, they
first described all maximal ideals of $R[x;\sigma,0]$, and then used their description to characterize quasi-duoness. In this article we work in the opposite direction. In fact, we use our main result (Theorem~\ref{MainResult}) to find all maximal ideals of $S=R[x;\delta]$ as well as the Jacobson radical of $R[x;\delta]$.

\begin{remark}\label{r4.1}
Suppose that $R[x;\delta]$ is quasi-duo. Using the same argument as in the proof of Theorem~\ref{MainResult}, for any maximal ideal $M$ of $R[x;\delta]$, if $M_0=M\cap R$, then $R/M_0$ is a field and $\delta(R)\subseteq M_0$.
Let $I$ be a maximal ideal of $R$ such that $R/I$ is a field and $\delta(R)\subseteq I$. Then the map
	\begin{displaymath}
		\Phi_{I}\colon R[x;\delta]\to (R/I)[x], \quad a_0+a_1x+\cdots+a_nx^n\mapsto \overline {a_0}+\overline{a_1}x+\cdots+\overline{a_n}x^n
	\end{displaymath}
	is a (well-defined) surjective ring morphism. Moreover, $\ker \Phi_I =I[x;\delta]$. Hence,
	\begin{displaymath}
		R[x;\delta]/I[x;\delta]\cong (R/I)[x].
	\end{displaymath}
In particular, $R[x;\delta]/I[x;\delta]$ is semiprimitive by Lemma~\ref{l3.1}.
These facts will be used several times in this section.
\end{remark}

\begin{prop}\label{p4.2}
	Let $R[x;\delta]$ be a quasi-duo differential polynomial ring,
	and let $I$ and $\Phi_I$ be as in Remark~\ref{r4.1}. Then for any maximal ideal $N$ of $S$ containing $I$, we have
	\begin{displaymath}
		\Phi_I^{-1}(\Phi_I(N))=N.
	\end{displaymath}
\end{prop}

\begin{proof}
Clearly, $N\subseteq \Phi_I^{-1}(\Phi_I(N))$.
Because of the maximality of $N$, we have either $N=\Phi_I^{-1}(\Phi_I(N))$ or $\Phi_I^{-1}(\Phi_I(N))=S$.
If $\Phi_I^{-1}(\Phi_I(N))=S$, then there exists $a_0+a_1x+\ldots+a_nx^n\in N$ such that
\begin{displaymath}
	\overline 1=\overline {a_0}+\overline{a_1}x+\ldots+\overline{a_n}x^n.
\end{displaymath}
Hence, $1-a_0,a_1,\ldots, a_n\in I$.
For any $i \geq 1$ we have $a_i \in I$ and hence $a_i x^i \in N$.
This implies that $a_0 \in N$.
From the fact that $1-a_0\in N$, we conclude that $1\in N$.
This is a contradiction. Hence, $N=\Phi_I^{-1}(\Phi_I(N))$.
\end{proof}

Given a differential polynomial ring $R[x;\delta]$, denote by $\Cal{M}(R)$ the set of maximal ideals $I$ of $R$ such that $R/I$ is a field and $\delta(R)\subseteq I$.
Recall that a \emph{monic polynomial} in $R[x;\delta]$ is an element whose highest degree coefficient is equal to $1$.
For any $I\in \Cal{M}(R)$, the set of all irreducible monic polynomials in $(R/I)[x]$ is denoted by $\Cal{P}(R/I)$.

\begin{theorem}\label{t4.4}
Let $S=R[x;\delta]$ be a quasi-duo differential polynomial ring.
Consider the following two maps:
	\begin{enumerate}[\rm (i)]
		\item Associate with any pair $A=(I, x^n+\overline{a_{n-1}}x^{n-1}+\cdots+\overline{a_0})\in (\Cal{M}(R), \Cal{P}(R/I))$ the maximal ideal $M(A)=I[x,\delta]+\langle x^n+{a_{n-1}}x^{n-1}+\cdots+{a_0}\rangle_S$ of $S$; 
		\item Associate with any maximal ideal $M$ of $S$ the pair
		\begin{displaymath}
			A(M)=(M_0, p(x))\in (\Cal{M}(R), \Cal{P}(R/M_0))
		\end{displaymath}
where $M_0=M\cap R$ and $p(x)\in (R/M_0)[x]$ is such that $\langle p(x)\rangle_{(R/M_0)[x]}=\Phi_{M_0}(M)$. 
	\end{enumerate}				
	Then, these maps yield two mutually inverse bijections between the set of all maximal ideals of $S$ and the set
	\begin{displaymath}
		\{(I,p(x))\mid I\in \Cal{M}(R) \text{ and } p(x)\in \Cal{P}(R/I) \}.
	\end{displaymath}
	\end{theorem}
	
\begin{proof}
We must first show that the maps from (i) and (ii) are well-defined, that is $M(A)$ is a maximal ideal of $S$ for any pair $A=(I,p(x))$, and $A(M)$ is an element of the set $\{(I,p(x))\mid I\in \Cal{M}(R) \text{ and } p(x)\in \Cal{P}(R/I) \}$ for any maximal ideal $M$ of $S$.
		
Let $A$ be a pair $(I, p(x))$ where $K=R/I$ is a field, $p(x)=x^n+\overline{a_{n-1}}x^{n-1}+\cdots+\overline{a_0}$ is an irreducible polynomial in $K[x]$. We have $L=\langle p(x)\rangle _{K[x]}$, i.e. the ideal of $K[x]$ generated by $p(x)$. Since $p(x)$ is irreducible, $L$ is maximal in $K[x]$. Notice that $\Phi_I(x^n+{a_{n-1}}x^{n-1}+\cdots+{a_0})=p(x)$.
Hence, if $ \Phi_I(f)\in L$, then $f\in\langle x^n+{a_{n-1}}x^{n-1}+\cdots+{a_0}\rangle_S+I[x;\delta]$
(using that $R[x;\delta]/I[x;\delta]\cong K[x]$ via $\Phi_I$ in Remark~\ref{r4.1}).
This implies that $\Phi_I^{-1}(L)=M(A)$.
Again, since $R[x;\delta]/I[x;\delta]\cong K[x]$ via $\Phi_I$, the ideal $M(A)$ is a maximal ideal of $R[x;\delta]/I[x;\delta]$.
Notice that $I[x;\delta]\subseteq M(A)$, and hence $M(A)$ is a maximal ideal of $S$. Thus, the map from (i) is well-defined.
		
Now assume that $M$ is a maximal ideal of $S$. Put $M_0=M\cap R$.
By Remark~\ref{r4.1}, $K=R/M_0$ is a field and $\delta(R)\subseteq M_0$.
Hence, $M_0\in \Cal{M}(R)$. Since $M$ is maximal in $S$, the ideal $\Phi_{M_0}(M)$ is also maximal in $K[x]$, by Remark~\ref{r4.1}. Therefore, there exists a unique irreducible monic polynomial $p(x)\in K[x]$ such that
		\begin{displaymath}
			\Phi_{M_0}(M)=\langle p(x)\rangle_{(R/M_0)[x]}.
		\end{displaymath}
Hence, the map from (ii) is well-defined.
		
Now we will show that $M(A(M))=M$ holds for any maximal ideal $M$ of $S$, and that $A(M(A))=A$ holds for any pair $A\in \{(I,p(x))\mid I\in \Cal{M}(R) \text{ and } p(x)\in \Cal{P}(R/I) \}$. Let $M$ be a maximal ideal of $S$. Assume that $A(M)=(M_0,p(x))$ where $M_0=M\cap R$ and $p(x)=x^n+\overline{a_{n-1}}x^{n-1}+\ldots+\overline{a_0}\in \Cal{P}(R/M_0)$. Notice that $M_0\subseteq M$, so that $M_0[x;\delta]\subseteq M$. Therefore, $x^n+{a_{n-1}}x^{n-1}+\ldots+{a_0}\in M$, which yields that $M_0[x;\delta]+\langle x^n+{a_{n-1}}x^{n-1}+\ldots+{a_0}\rangle_S\subseteq M$. Equivalently, $M(A(M))\subseteq M$. By the maximality of $M(A(M))$ in $S$ we get $M(A(M))=M$.
		
Let $A=(I, p(x))$ where $I$ is an ideal in $\Cal{M}(R)$ and $p(x)=x^n+\overline{a_{n-1}}x^{n-1}+\ldots+\overline{a_0}\in \Cal{P}(R/I)$. Then,
\begin{displaymath}
	M(A)=I[x;\delta]+\langle x^n+{a_{n-1}}x^{n-1}+\ldots+{a_0}\rangle_S.
\end{displaymath}
It is clear that $I=M(A)\cap R$ and that $\Phi_I(M(A))=\langle p(x)\rangle_{K[x]}$. Thus, $A=A(M(A))$.		
\end{proof}

We will now use the preceding theorem to describe the Jacobson radical of $R[x;\delta]$ and obtain a result
which resembles \cite[Theorem 1.2]{Pa_BeMaSh_16}.

\begin{corollary}
Let $S=R[x;\delta]$ be a quasi-duo differential polynomial ring.
Put ${K=\bigcap\limits_{I\in \Cal{M}(R)}I}$.
Then the Jacobson radical of $S$ is 
$J(S)=K[x;\delta]=(J(S)\cap R)[x;\delta]$.
\end{corollary}
\begin{proof}
	By Theorem~\ref{t4.4}, $J(S)$ is the intersection of all ideals of the form
\begin{displaymath}
	I[x;\delta]+\langle x^n+{a_{n-1}}x^{n-1}+\ldots+{a_0}\rangle_S
\end{displaymath}
where $I$ ranges over $\Cal{M}(R)$ and the polynomial $x^n+\overline{a_{n-1}}x^{n-1}+\ldots+\overline{a_0}$ ranges over $\Cal{P}(R/I)$.
That is,
\begin{displaymath}
	J(S)= \bigcap\limits_{I\in \Cal{M}(R)}{ \bigcap\limits_{x^n+\overline{a_{n-1}}x^{n-1}+\ldots+\overline{a_0}\in \Cal{P}(R/I)}\left( I[x,\delta]+\langle x^n+\ldots+{a_0}\rangle_S \right)}.
\end{displaymath}
For any $I\in \Cal{M}(R)$, we define
\begin{displaymath}
	L_I=\bigcap\limits_{x^n+\overline{a_{n-1}}x^{n-1}+\ldots+\overline{a_0}\in \Cal{P}(R/I)} (I[x;\delta]+\langle x^n+a_{n-1}x^{n-1}+\ldots+{a_0}\rangle_S).
\end{displaymath}
According to the maps in Theorem~\ref{t4.4}, when $x^n+\overline{a_{n-1}}x^{n-1}+\ldots+\overline{a_0}$ ranges over $\Cal{P}(R/I)$, then $I[x;\delta]+\langle x^n+a_{n-1}x^{n-1}+\ldots+{a_0}\rangle_S$ ranges over the set of all maximal ideals of $R[x;\delta]/I[x;\delta]$. Hence, $L_I$ is the Jacobson radical of $R[x;\delta]/I[x;\delta]$. Thus, $L_I=I[x;\delta]$ using that $R[x;\delta]/I[x;\delta]$ is semiprimitive. Therefore,
$J(S)=K[x;\delta]=(J(S)\cap R)[x;\delta]$.
\end{proof}

\section{Differential polynomial rings in several indeterminates}\label{Sec:SeveralVariables}

In this section we shall show that differential polynomial rings in several indeterminates can never be quasi-duo (see Theorem~\ref{thm:several}).

Let us begin by recalling the definition of a differential polynomial ring in a set of indeterminates. Let $I$ be a non-empty (possibly infinite) countable set, let $D=\{\delta_i\mid i\in I\}$ be a family of
derivations on $R$ (by ``a family" we mean that all $\delta_i$'s need not be distinct), and let $X=\{x_i\mid i\in I\}$ be a set of distinct non-commuting indeterminates.
Given $R$, $D$ and $X$, we can define the ring $R[X;D]$ which is the set of all polynomials in the indeterminates $x_i\in X$ with coefficients from $R$. The addition in $R[X;D]$ is the natural one
and the multiplication is generated by the commutation rule $x_ia=ax_i+\delta_i(a)$, for $i\in I$. The ring $R[X, D]$ is called a \emph{differential polynomial ring in several indeterminates}.
Readers are referred to \cite{Pa_Bu_80,Pa_TsWuCh_07} for more details on this class of rings.
In particular, every element $f \in R[X;D]$ can be written in the form
\begin{displaymath}
	f=a_1t_1+a_2t_2+\cdots+a_nt_n,
\end{displaymath}
where $a_1,a_2,\cdots, a_n\in R\setminus \{0\}$ and $t_1,t_2,\ldots,t_n$ are distinct monomials in $X$, i.e. finite words in the alphabet $X$. In this case, the support of $f$ is defined as $\supp(f)=\{t_1,t_2,\ldots,t_n\}$. If $\delta_i=0$, for all $i\in I$, then $R[X]$ is called a \textit{free polynomial ring}.\\

A subring $B$ of a ring $A$ is called a \emph{corner subring} of $A$
if $B$ is unital, possibly with $1_A \neq 1_B$,
and if there exists an additive subgroup $C$ of $A$ such that $A=B\oplus C$
and $BC, CB \subseteq C$.
The subgroup $C$ is called a \emph{complement} of $B$.
We say that $B$ is a \emph{left corner} of $A$ if the complement $C$ satisfies $BC \subseteq C$ only.
The notion of a \emph{right corner} is defined analogously.
A classical example of a corner subring is given by $B=eAe$, where $e \in A$ is an idempotent.

The following lemma shows that quasi-duoness
of differential polynomial rings in several indeterminates can be inherited by certain subrings.

\begin{lemma}\label{l2.1}
Let $R$ be a ring, let $I$ be a non-empty countable set,
let $D=\{\delta_i\mid i\in I\}$ be a family of derivations on $R$,
and let $X=\{x_i\mid i\in I\}$ be a set of non-commuting indeterminates.
For any subset $J \subseteq I$, put $X_J=\{x_i\mid i\in J\}$ and $D_J=\{\delta_i\mid i\in J \}$.
The ring $S_J=R[X_J;D_J]$ is a right (left) corner of $S=R[X;D]$.
In particular, if $S$ is left (right) quasi-duo, then $S_J$ is left (right) quasi-duo for any subset $J \subseteq I$.
\end{lemma}

\begin{proof}
Take $J \subseteq I$.
Denote by $X_J^+$ the set of all nontrivial monomials in $X_J$, that is $X_J^+$ is the set of all ``words" $x_{j_1}^{m_1}x_{j_2}^{m_2}\cdots x_{j_t}^{m_t}$ where $x_{j_i}$ ranges over $X_j$, and $t,m_i>0$. Now put
\begin{displaymath}
	C=\{\,f\in S\mid \supp(f)\cap X_J^+=\emptyset\,\}.
\end{displaymath}
To demonstrate that $S_J$ is a right corner of $S$, we will show $C$ is an additive group, $S_J\oplus C=S$ and $CS_J\subseteq C$. It is easy to show the first statement since $C$ is an additive subgroup of $S$ generated by all elements of the following form $ax_{i_1}^{t_1}x_{i_2}^{t_2}\cdots x_{i_m}^{t_m}\in S$ with $x_{i_1}^{t_1}x_{i_2}^{t_2}\cdots x_{i_m}^{t_m}\notin X_J^+$.
Again, by the definitions of $S_J$ and $C$, one has $S=S_J+C$ and $S_J\cap C=\{0\}$. Therefore, $S=S_J\oplus C$.
Now we must show that $CS_J\subseteq C$. If $f=x_{i_1}^{t_1}x_{i_2}^{t_2}\cdots x_{i_m}^{t_m}\in X_J^+$ and $g=x_{j_1}^{q_1}x_{j_2}^{q_2}\cdots x_{j_l}^{q_l}\notin X_J^+$, then $gf\in C$. Notice that $S_J$ is the set of all finite sums $\alpha f$ where $\alpha \in R$ and $f\in X_J^+$, and that $C$ is the set of all finite sums $\beta g$ where $\beta \in R$ and $g\notin X_J^+$. Thus, $CS_J\subseteq C$.
The last conclusion now follows directly from \cite[Theorem 1.2]{Pa_LeMaPu_08}.

Analogously, one can show that $S_J$ is a left corner of $S$ and that right quasi-duoness of $S$ implies right quasi-duoness of $S_J$.
\end{proof}

\begin{lemma}\label{St2.2}
Let $I$ be a non-empty countable set and let $S=R[X;D]$ be a left (right) quasi-duo differential polynomial ring in several indeterminates (as above).
If $R$ is a simple ring, then $R$ is a field, $\delta=0$ and $|I|=1$.
\end{lemma}

\begin{proof}
Suppose that $S$ is left quasi-duo. (The right quasi-duo case can be treated analogously and is therefore omitted.)
Take $i\in I$. By Lemma~\ref{l2.1}, $S_i=K[x_i;\delta_i]$ is left quasi-duo.
Using Lemma~\ref{t2.2} we conclude that $R$ is a field and that $\delta_i=0$.

It remains to show that $|I|=1$. If $|I|>1$, then it is well-known that the free polynomial algebra $K[X]$ is left primitive (see e.g. \cite[Page 36]{Bo_Ja_64}). In view of \cite[Proposition 4.1]{Pa_LaDu_05}, $K[X]$ is a division ring. This is a contradiction. Therefore, $|I|=1$.
\end{proof}

\begin{theorem}\label{thm:several}
Let $I$ be a non-empty countable set and let $S=R[X;D]$ be a differential polynomial ring in several indeterminates (as above).
If $S$ is left (right) quasi-duo, then $|I|=1$.
\end{theorem}

\begin{proof}
The proof is essentially the same as the proof of Proposition~\ref{Prop:Nec}, and we will therefore omit some details.
As in the proof of Proposition~\ref{Prop:Nec} we need to consider two cases:

{\bf Case 1.} This will lead to that $R$ is a division ring. Thus, by Lemma~\ref{St2.2}, we get that $|I|=1$.

{\bf Case 2.} Analogously to the proof of Proposition \ref{Prop:Nec}
we will define a surjective ring morphism
\begin{displaymath}
		\varphi :  R[X;D]\to (R/M_0)[X;\overline{D}],
		\quad a_0+a_1t_1+\cdots+a_nt_n\mapsto \overline {a_0}+\overline{a_1}t_1+\cdots+\overline{a_n}t_n
\end{displaymath}
and use it to conclude that $R/M_0[X;\overline{D}]$ is left (right) quasi-duo.
By invoking case 1, we conclude that $|I|=1$.
\end{proof}


\begin{thebibliography}{99}
	
	\bibitem{Pa_BeMaSh_16}
	J. P. Bell, B. W. Madill and F. Shinko,
	Differential polynomial rings over rings satisfying a polynomial identity,
	{\it J. Algebra} \textbf{423} (2015), 28--36.
		
	\bibitem{Pa_Bu_80}
	V. D. Burkov,
	Differentially prime rings (Russian),
	{\it Uspekhi Mat. Nauk} \textbf{35} (1980), 219--220.
		
	\bibitem{Feller_58}
	E. H. Feller,
	Properties of primary noncommutative rings,
	{\it Trans. Amer. Math. Soc.} \textbf{89} (1958), 79--91.
	
	\bibitem{Bo_Ja_64}
	N. Jacobson,
	\textit{Structure of Rings},
	Revised edition, AMS Colloquium Publications, Vol. 37, Providence (1964).
	
	\bibitem{Bo_La_91}
	T. Y. Lam,
	{\it A First Course in Noncommutative Rings},
	Graduate Texts in Mathematics, Vol. 131, Springer-Verlag, Berlin (1991).
	
	\bibitem{Pa_LaDu_05}
	T. Y. Lam and A. S. Dugas,
	Quasi-duo rings and stable range descent,
	{\it J. Pure Appl. Algebra} \textbf{195}(3) (2005), 243--259.

	\bibitem{Pa_LeMaPu_08}
	A. Leroy, J. Matczuk and E. R. Puczyłowski,
	Quasi-duo skew polynomial rings,
	{\it J. Pure Appl. Algebra} \textbf{212}(8) (2008), 1951--1959.
	
	\bibitem{Le10}
	A. Leroy, J. Matczuk and. E. R. Puczyłowski,
	A description of quasi-duo $\Z$-graded rings,
	{\it Comm. Algebra} \textbf{38}(4) (2010), 1319--1324.
	
	\bibitem{Ny13}
	P. Nystedt,
	A combinatorial proof of associativity of Ore extensions,
	{\it Discrete Math.} \textbf{313}(23) (2013), 2748--2750.
	
	\bibitem{Pu06}
	E. R. Puczyłowski,
	Questions related to Koethe's nil ideal problem,
	Algebra and its applications, 269--283,
	{\it Contemp. Math.}, \textbf{419}, Amer. Math. Soc., Providence, RI, (2006).
	
	\bibitem{Pu15}
	E. R. Puczyłowski,
	On quasi-duo rings,
	Noncommutative rings and their applications, 243--252,
	{\it Contemp. Math.}, \textbf{634}, Amer. Math. Soc., Providence, RI, (2015).
	
	\bibitem{Pa_TsWuCh_07}
	Y.-T. Tsai, T.-Y. Wu and C.-L. Chuang,
	Jacobson radicals of Ore extensions of derivation type,
	{\it Comm. Algebra} \textbf{35}(3) (2007), 975--982.
	
	\bibitem{Yu95}
	H.-P. Yu,
	On quasi-duo rings,
	{\it Glasgow Math. J.} \textbf{37}(1) (1995), 21--31.
	
\end{thebibliography}
\end{document}